\documentclass{amsart}      
\usepackage{amssymb,amsthm, amsmath, amsfonts} 
\usepackage{graphics}                 
\usepackage{hyperref}                 
\usepackage[all]{xy}
\usepackage{enumerate}

\newtheorem{theorem}{Theorem}[section]
\newtheorem{lemma}[theorem]{Lemma}
\newtheorem{proposition}[theorem]{Proposition}

\newtheorem{corollary}[theorem]{Corollary}
\newtheorem{definition}[theorem]{Definition}
\newtheorem{example}[theorem]{Example}

\numberwithin{equation}{section}

\DeclareMathOperator{\rad}{rad}
\DeclareMathOperator{\Ob}{Ob}

\DeclareMathOperator{\Hom}{Hom}
\DeclareMathOperator{\End}{End}
\DeclareMathOperator{\Ext}{Ext}
\DeclareMathOperator{\gldim}{gldim}

\title{Representations of modular skew group algebras}
\author{Liping Li}
\thanks{The author would like to thank Sarah Witherspoon, Peter Webb, Fei Xu, and Guodong Zhou for their helps. Witherspoon and Zhou discussed the Koszul properties of skew group algebras with the author during a conference held at University of Washington, Seattle, August 2012. In a personal communication Xu proposed to classify the representation types of transporter categories. The author was motivated to study representations of modular skew group algebras systematically. In the numerous communications Witherspoon, Webb and Xu provided many invaluable suggestions on approaches, proofs, and references. Without these helps this paper would not have appeared.}
\address{Department of Mathematics, University of California, Riverside, CA, 92507.}
\email{lipingli@math.ucr.edu}
\begin{document}

\begin{abstract}
In this paper we study representations of skew group algebras $\Lambda G$, where $\Lambda$ is a connected, basic, finite-dimensional algebra (or a locally finite graded algebra) over an algebraically closed field $k$ with characteristic $p \geqslant 0$, and $G$ is an arbitrary finite group each element of which acts as an algebra automorphism on $\Lambda$. We characterize skew group algebras with finite global dimension or finite representation type, and classify the representation types of transporter categories for $p \neq 2,3$. When $\Lambda$ is a locally finite graded algebra and the action of $G$ on $\Lambda$ preserves grading, we show that $\Lambda G$ is a generalized Koszul algebra if and only if so is $\Lambda$.
\end{abstract}

\maketitle

\section{Introduction}

Let $\Lambda$ be a a connected, basic, finite-dimensional algebra over an algebraically closed filed $k$ with characteristic $p \geqslant 0$, and $G$ is an arbitrary finite group each element of which acts as an algebra automorphism on $\Lambda$. The action of $G$ on $\Lambda$ determines a finite-dimensional algebra $\Lambda G$, called \textit{skew group algebra} of $\Lambda$. Skew group algebras include ordinary group algebras as special examples when the action of $G$ on $\Lambda$ is trivial.

Representations and homological properties of skew group algebras (or more generally, group-graded algebras and crossed products) have been widely studied, see \cite{Boisen,Cohen,Marcus,Passman,Reiten,Witherspoon,Zhong}. When $|G|$, the order of $G$, is invertible in $k$, it has been shown that $\Lambda G$ and $\Lambda$ share many common properties. For example, $\Lambda G$ is of finite representation type (self-injective, of finite global dimension, an Auslander algebra) if and only so is $\Lambda$ (\cite{Reiten}); if $\Lambda$ is piecewise hereditary, so is $\Lambda G$ (\cite{Dionne}). If $\Lambda$ is a positively graded algebra and the action of $G$ preserves grading, then $\Lambda G$ is a Koszul algebra if and only if so is $\Lambda$, and the Yoneda algebra of $\Lambda G$ is also a skew group algebra of the Yoneda algebra of $\Lambda$ (\cite{Martinez}). However, when $G$ is a modular group, many of the above results fail. Therefore, it makes sense to ask under what conditions $\Lambda G$ and $\Lambda$ still share these properties.

As for ordinary group algebras, it is natural to use the induction-restriction procedure and consider the relatively projective property of a $\Lambda G$-module $M$ with respect to the skew group algebra $\Lambda S$, where $S$ is a Sylow $p$-subgroup of $G$. This theory has been described in \cite{Boisen,Marcus}. Using some elementary arguments, we show that $\Lambda G$ is of finite representation type or of finite global dimension if and only if so is $\Lambda S$. Therefore, we turn to study the representations of $\Lambda S$.

Suppose that $\Lambda$ has a set of primitive orthogonal idempotents $E = \{ e_i \} _{i = 1}^n$ such that $\sum _{i=1}^n e_i = 1$ and $E$ is closed under the action of $S$. This condition holds in many cases, for example, when $G$ acts on a finite quiver, a finite poset, or a finite category, etc, and $\Lambda$ is the corresponding path algebra, incidence algebra, or category algebra. Under this assumption, we show that if $\Lambda S$ is of finite representation type (for the situation that $p \geqslant 5$ and $\Lambda$ is not a local algebra) or of finite global dimension, then the action of $S$ on $E$ must be free. In this situation, $\Lambda S$ is the matrix algebra over $\Lambda ^S$, the subalgebra of $\Lambda$ constituted of all elements in $\Lambda$ fixed by the action of $S$. Therefore, $\Lambda S$ is Morita equivalent to $\Lambda ^S$, and we can prove:

\begin{theorem}
Let $\Lambda$ be a basic finite-dimensional $k$-algebra and $G$ be a finite group acting on $\Lambda$. Let $S \leqslant G$ be a Sylow $p$-subgroup. Suppose that there is a set of primitive orthogonal idempotents $E = \{ e_i \} _{i = 1}^n$ such that $E$ is closed under the action of $S$ and $\sum _{i=1}^n e_i = 1$. Then:
\begin{enumerate}
\item $\gldim \Lambda G < \infty$ if and only if $\gldim \Lambda < \infty$ and $S$ acts freely on $E$. Moreover, if $\gldim \Lambda G < \infty$, then $\gldim \Lambda G = \gldim \Lambda = \gldim \Lambda^S$.
\item $\Lambda G$ is an Auslander algebra if and only if so is $\Lambda$ and $S$ acts freely on $E$.
\item Suppose that $p \neq 2, 3$ and $\Lambda$ is not a local algebra. Then $\Lambda G$ is of finite representation type if and only if so is $\Lambda^S$ and $S$ acts freely on $E$. If $\Lambda \cong \Lambda^S \oplus B$ as $\Lambda^S$-bimodules, $\Lambda G$ has finite representation type if and only if so does $\Lambda$ and $S$ acts on $E$ freely.
\end{enumerate}
\end{theorem}

Let $\mathcal{P}$ be a finite connected poset on which every element in $G$ acts as an automorphism. The Grothendick construction $\mathcal{T} = G \propto \mathcal{P}$ is called a \textit{transporter category}. It is a \textit{finite EI category}, i.e., every endomorphism in $\mathcal{T}$ is an isomorphism. Representations of transporter categories and finite EI categories have been studied in \cite{Li1,Li2,Webb1,Webb2,Xu1,Xu2}. In a paper \cite{Xu2}, Xu showed that the category algebra $k\mathcal{T}$ is a skew group algebra of the incidence algebra $k\mathcal{P}$. We study its representation type for $p \neq 2, 3$, and it turns out that we can only get finite representation type for very few cases:

\begin{theorem}
Let $G$ be a finite group acting on a connected finite poset $\mathcal{P}$ and suppose that $p \neq 2, 3$. Then the transporter category $\mathcal{T} = G \propto \mathcal{P}$ is of finite representation type if and only if one of the following conditions is true:
\begin{enumerate}
\item $|G|$ is invertible in $k$ and $\mathcal{P}$ is of finite representation type;
\item $\mathcal{P}$ has only one element and $G$ is of finite representation type.
\end{enumerate}
\end{theorem}

When $\Lambda$ is a locally finite graded algebra and the action of $G$ preserves grading, the skew group algebra $\Lambda G$ has a natural grading. Therefore, it is reasonable to characterize its Koszul property. Since $|G|$ might not be invertible in $k$, the degree $0$ part of $\Lambda G$ might not be semisimple, and the classical Koszul theory cannot apply. Thus we use a generalized Koszul theory developed in \cite{Li3,Li4}, which does not demand the semisimple property. Using the generalized definition of Koszul algebras, we prove that $\Lambda G$ is a generalized Koszul algebra if and only if so is $\Lambda$. Moreover, a careful check tells us that the argument in Martinez's paper \cite{Martinez} works as well in the general situation, so we establish the following result.

\begin{theorem}
Let $M$ be a graded $\Lambda G$-module and let $\Gamma = \Ext _{\Lambda} ^{\ast} (M, M)$. Then:
\begin{enumerate}
\item The $\Lambda G$-module $M \otimes_k kG$ is a generalized Koszul $\Lambda G$-module if and only if as a $\Lambda$-module $M$ is generalized Koszul. In particular, $\Lambda G$ is a generalized Koszul algebra over $(\Lambda G)_0 = \Lambda_0 \otimes_k kG$ if and only if $\Lambda$ is a generalized Koszul algebra over $\Lambda_0$.
\item If as a $\Lambda$-module $M$ is generalized Koszul, then $\Ext _{\Lambda G} ^{\ast} (M \otimes_k kG, M \otimes_k kG)$ is isomorphic to the skew group algebra $\Gamma G$ as graded algebras.
\end{enumerate}
\end{theorem}

This paper is organized as follows. In Section 2 we describe the induction-restriction procedure and use it to show that $\Lambda G$ and $\Lambda S$ share some common properties. In Section 3 we study representations of the skew group algebra $\Lambda S$ and prove the first theorem. Transporter categories and its representation type are considered in Section 4, where we prove the second theorem. In the last section we briefly describe a generalized Koszul theory and apply it to characterize the Koszul property of $\Lambda G$ when $\Lambda$ is graded and the action of $G$ preserves grading.

We introduce the notation and convention here. For $\lambda \in \Lambda$ and $g \in G$, we denote the image of $\lambda$ under the action of $g$ by $^g\lambda$ or $g(\lambda)$. The skew group algebra $\Lambda G$ as a vector space is $\Lambda \otimes_k kG$, and the multiplication is defined by $(\lambda \otimes g) \cdot (\mu \otimes h) = \lambda (^g\mu) \otimes gh$. To simplify the notation, we denote $\lambda \otimes g$ by $\lambda g$. Correspondingly, the multiplication is $(\lambda g) \cdot (\mu h) = \lambda (^g\mu) gh$ or $\lambda g(\mu) gh$. The identities of $\Lambda$ and $G$ are denoted by $1$ and $1_G$ respectively.

All modules we consider in this paper are left finitely generated modules (or left locally finite graded modules when $\Lambda$ is a locally finite graded algebra). Composition of maps, morphisms and actions is from right to left. To simplify the expression of statements, we view the zero module as a projective or a free module. When $p = 0$, the Sylow $p$-group is the trivial group.

\section{Induction and Restriction}

Let $H \leqslant G$ be a subgroup of $G$. Then $\Lambda H$ is a subalgebra of $\Lambda G$. As for group algebras, we define induction and restriction. Explicitly, for a $\Lambda H$ module $V$, the induced module is $V \uparrow _H^G = \Lambda G \otimes _{\Lambda H} V$, where $\Lambda G$ acts on the left side. Every $\Lambda G$-module $M$ can be viewed as a $\Lambda H$-module. We denote this restricted module by $M \downarrow _H^G$. Observe that $\Lambda G$ is a free $\Lambda H$-module. Therefore, these two functors are exact, and perverse projective modules.

The following property is well known, see for example Lemmas 2.3 and 2.4 in \cite{Boisen} and \cite{ARS}.

\begin{proposition}
Let $V$ and $M$ be a $\Lambda H$-module and a $\Lambda G$-module respectively. Then:
\begin{enumerate}
\item $V$ is a summand of $V \uparrow _H^G \downarrow _H^G$;
\item if $|G:H|$ is invertible in $k$, then $M$ is a summand of $M \downarrow _H^G \uparrow _H^G$.
\end{enumerate}
\end{proposition}

The above proposition immediately implies:

\begin{corollary}
Let $H \leqslant G$ be a subgroup of $G$ and $M$ be a $\Lambda G$-module.
\begin{enumerate}
\item If $\Lambda G$ is of finite representation type, so is $\Lambda H$, and in particular, $\Lambda$ is of finite representation type.
\item If the global dimension $\gldim \Lambda G < \infty$, then $\gldim \Lambda H < \infty$, and in particular, $\gldim \Lambda < \infty$.
\item If $Q$ is an injective $\Lambda G$-module, then $Q \downarrow _H^G$ is an injective $\Lambda H$-module.
\end{enumerate}
Suppose that $|G: H|$ is invertible. Then:
\begin{enumerate} \setcounter{enumi}{3}
\item$\Lambda G$ is of finite representation type if and only if so is $\Lambda H$.
\item $M$ is a projective $\Lambda G$-module if and only if $M \downarrow _H^G$ is a projective $\Lambda H$-module. In particular, $\gldim \Lambda G = \gldim \Lambda H$.
\item $Q$ is an injective $\Lambda G$-module if and only if $Q \downarrow _H^G$ is an injective $\Lambda H$-module.
\end{enumerate}
\end{corollary}

\begin{proof}
The proof of these facts are straightforward. The reader can check them by using the same technique for regular group algebras. We only show (3) and (6).

(3): The $\Lambda H$-module $Q \downarrow _H^G $ is injective if and only if the functor $\Hom _{\Lambda H} (-, Q \downarrow _H^G)$ is exact. By the Nakayama relations, this is true if and only if $\Hom _{\Lambda G} (- \uparrow _H^G, Q)$, the composite of the induction functor and $\Hom _{\Lambda G} (-, Q)$, is exact, which is obvious since both functors are exact.

Now suppose that $n = |G:H|$ is invertible in $k$.

(6): The $\Lambda G$-module $Q$ is injective if and only if $\Hom _{\Lambda G} (-, Q)$ is exact. Note that this functor is a summand of the functor $\Hom _{\Lambda G} (- \downarrow _H^G \uparrow _H^G, Q)$ since $|G: H|$ is invertible in $k$. Therefore, it suffices to show the exactness of $\Hom _{\Lambda G} (- \downarrow _H^G \uparrow _H^G, Q)$. By the Nakayama relations, it is isomorphic to $\Hom _{\Lambda H} (- \downarrow _H^G, Q \downarrow _H^G)$, the composite functor of $\downarrow _H^G$ and $\Hom _{\Lambda H} (-, Q \downarrow _H^G)$. The conclusion follows from the observation that the restriction functor is exact and $Q \downarrow _H^G$ is injective.
\end{proof}

In the case that $kG$ is semisimple, we get the following well known result.

\begin{corollary}
If $|G|$ is invertible in $k$, then $\Lambda G$ is of finite representation type if and only if so is $\Lambda$; a $\Lambda G$-module $M$ is projective if and only if it is a projective $\Lambda$-module; and $\gldim \Lambda G = \gldim \Lambda$.
\end{corollary}

\begin{proof}
Let $S = 1$ and apply the previous corollary.
\end{proof}

\section{Actions of $p$-groups}

From the previous section we know that the action of a Sylow $p$-subgroup $S \leqslant G$ plays a crucial role in determining the representation type and the global dimension of $\Lambda G$. Explicitly, $\Lambda G$ and $\Lambda S$ have the same representation type and the same global dimension. Throughout this section let $S$ be a $p$-group. We suppose that there is set $E = \{ e_i \} _{i=1}^n$ of primitive orthogonal idempotents in $\Lambda$ which is \textit{complete} (i.e., $\sum_{i=1}^n e_i = 1$) and is \textit{closed} under the action of $S$. Although this is not always true, in many cases the action of $S$ on $\Lambda$ is induced by an action of $S$ on some discrete structures such as quivers, posets, etc, and this condition is satisfied.

\begin{lemma}
The set $E$ is also a complete set of primitive orthogonal idempotents of $\Lambda S$, where we use $e_i$ rather than $e_i 1_S$ to simplify the notation.
\end{lemma}

\begin{proof}
It suffices to show that every element in $E$ is primitive as an idempotent in $\Lambda S$. This is equivalent to the statement that a (and hence every) complete set of primitive orthogonal idempotents in $\Lambda S$ has precisely $n$ elements.

Let $\mathfrak{r}$ be the radical of $\Lambda$ and let $\overline{\Lambda} = \Lambda / \mathfrak{r}$. Then $\mathfrak{r} \otimes_k kS$ is a two-sided ideal of $\Lambda S$ since $\mathfrak{r}$ is a two-sided ideal of $\Lambda$ and every element $g \in S$ sends $\mathfrak{r}$ into $\mathfrak{r}$. Moreover, this ideal is nilpotent, and hence is contained in the radical of $\Lambda S$. Therefore, a complete set of primitive orthogonal idempotents in $\Lambda S$ and a complete set of primitive orthogonal idempotents in $\Lambda S / \mathfrak{r} S = (\Lambda \otimes_k kS) / (\mathfrak{r} \otimes_k kS) \cong \overline{\Lambda} \otimes_k kS = \overline{\Lambda} S$ (which is also a skew group algebra) have the same cardinal. Since elements in $\overline{E} = \{ \overline{e}_i \} _{i=1}^n$ are nonzero pairwise orthogonal idempotents in $\overline{\Lambda} S$ whose sum equals $1_{\overline{\Lambda} S}$, where $\overline{e}_i$ is the image of $e_i$ in the quotient algebra, the conclusion follows if each $\overline{e}_i$ is primitive. This is true if and only if $\overline{e}_i (\overline{\Lambda} S) \overline{e}_i$ is a local algebra.

Note that $\overline{\Lambda} \cong \oplus_{i=1}^n k \overline{e}_i$. Therefore, for a fixed $i$, we have
\begin{equation*}
\overline{e}_i (\overline{\Lambda} S) \overline{e}_i = \overline{e}_i (kS) \overline{e}_i =\bigoplus _{g \in S} k \overline{e}_i g \overline{e}_i = \bigoplus _{g \in S} k \overline{e}_i \overline{g(e_i)} g.
\end{equation*}
Let $H = \{ g \in S \mid g(e_i) = e_i \}$. Then $H$ is a subgroup of $S$. We have:
\begin{equation*}
\overline{e}_i (\overline{\Lambda} S) \overline{e}_i = \bigoplus _{g \in S} k \overline{e}_i \overline{g(e_i)} g = \bigoplus _{g \in H} k \overline{e}_i g \cong kH,
\end{equation*}
where $kH$ is a regular group algebra since every $g \in H$ fixes $\overline{e}_i$. Clearly, $kH$ is local since as a subgroup of $S$, $H$ is a $p$-group as well. This finishes the proof.
\end{proof}

From this lemma we know that every indecomposable $\Lambda S$-module can be described as follows (up to isomorphism):
\begin{equation*}
\Lambda S e_i = \bigoplus _{g \in S} \Lambda g(e_i) \otimes_k g, \quad 1 \leqslant i \leqslant n.
\end{equation*}

Since $E$ is closed under the action of $S$, we take $\{e_1, e_2, \ldots, e_s \}$ to be a set of representatives of $S$-orbits, and set $\epsilon = \sum _{i=1}^s e_i$. This is an idempotent of $\Lambda$, and hence an idempotent of $\Lambda S$.

The next lemma is described in \cite{Cohen} for the situation that $S$ acts freely on $E$, see Lemma 3.1 in that paper. We observe that the result is still true even if the action is not free, and gives a sketch of the proof.

\begin{lemma}
Let $S, \Lambda, E$ and $\epsilon$ be as defined before. Then $\Lambda S$ is Morita equivalent to $\epsilon \Lambda S \epsilon$.
\end{lemma}

\begin{proof}
Consider the associated $k$-linear category $\mathcal{A}$ of $\Lambda S$ defined by Gabriel's construction (\cite{Bautista}). Note that two idempotents in the same $S$-orbit are isomorphic when viewed as objects in $\mathcal{A}$. Therefore, $\Lambda S$ is Morita equivalent to the skeletal category of $\mathcal{A}$, and this skeletal category is Morita equivalent to $\epsilon \Lambda S \epsilon$ by Gabriel's construction again.
\end{proof}

Although we assumed in this lemma that $S$ is a $p$-group, the conclusion holds for arbitrary groups.

The following proposition motivates us to study the situation that $S$ acts freely on $E$. That is, for every $1 \leqslant i \leqslant n$ and $g \in S$, $g(e_i) \neq e_i$.

\begin{proposition}
Let $S, \Lambda$ and $E$ be as above.
\begin{enumerate}
\item If $\gldim \Lambda S < \infty$, then the action of $S$ on $E$ is free.
\item Suppose that $p \neq 2, 3$ and $\Lambda$ is not a local algebra. If $\Lambda S$ is of finite representation type, then $S$ acts freely on $E$.
\end{enumerate}
\end{proposition}

\begin{proof}
We use contradiction to prove the first part. Without loss of generality we assume that there is some $1_S \neq g \in S$ such that $g(e_1) = e_1$. Consider the subgroup $H = \langle g \rangle \leqslant S$. By the assumption, $H$ is a nontrivial $p$-group. By Corollary 2.2, we only need to show that $\gldim \Lambda H = \infty$.

It is obvious that $\Lambda He_1$ is an indecomposable projective $\Lambda H$-module. Moreover, since $H = \langle g \rangle$ and $g$ fixes $e_1$, we actually have $\Lambda He_1 = \Lambda e_1 H = \Lambda e_1 \otimes_k kH$ as $\Lambda H$-modules, where $h \in H$ acts diagonally on $\Lambda e_1 \otimes_k kH$. Consider the $\Lambda H$-module $\Lambda e_1 \cong \Lambda e_1 \otimes_k k$, where $h \in H$ acts on $k$ trivially and sends $\lambda e_1 \in \Lambda e_1$ to $h (\lambda) e_1$. It has a projective cover $\Lambda e_1 \otimes_k kH$ and the first syzygy is isomorphic to $\Lambda e_1 \otimes_k \mathfrak{r}$, where $\mathfrak{r}$ is the radical of $kH$. Clearly, the first syzygy is not projective and its projective cover is isomorphic to $\Lambda e_1 \otimes_k kH$. Observe that $\Lambda e_1 \otimes_k k$ appears again as the second syzygy. Consequently, $\Lambda e_1 \otimes_k k$ has infinite projective dimension, so $\gldim \Lambda H = \infty$. This contradiction establishes the conclusion.

Now we prove the second part by contradiction. Since it is obviously true if $S$ is a trivial group (for example, if $p = 0$), without loss of generality we assume $S \neq 1$, so $p \geqslant 5$. Suppose that there are a primitive idempotent $e_i$ and some $1_S \neq g \in S$ such that $g (e_i) = e_i$. We want to show that $\Lambda S$ is of infinite representation type. Let $H$ be the cyclic group generated by $g$ and consider $\Lambda H$. By Corollary 2.2, it suffices to prove that $\Lambda H$ is of infinite representation type.

Since $\Lambda$ is a connected algebra and is not local, we can find another primitive idempotent $e_j$ connected to $e_i$, i.e., either $e_i \Lambda e_j \neq 0$ or $e_j \Lambda e_i \neq 0$. Consider the algebra $\Gamma = (e_i + e_j) \Lambda H (e_i + e_j)$. By Proposition 2.5 in page 35 of \cite{Auslander}, $\Gamma$-mod is equivalent to a subcategory of $\Lambda H$-mod. Therefore, we only need to show the infinite representation type of $\Gamma$.

Note that $\Lambda H e_i$ and $\Lambda H e_j$ are not isomorphic as $\Lambda H$-modules. Indeed, if they are isomorphic as $\Lambda H$-modules, then viewed as $\Lambda$-modules they must be isomorphic as well. But $_\Lambda \Lambda H e_i \cong (\Lambda e_i) ^{|H|}$ since $H$ fixes $e_i$. On the other hand, $_{\Lambda} \Lambda H e_j \cong \bigoplus _{h \in H} \Lambda h(e_j)$. For every $h \in H$, $h(e_j) \neq e_i$ since otherwise we should have $h^{-1} (e_i) = e_j$. This is impossible since $H$ fixes $e_i$. Therefore, $_\Lambda \Lambda H e_i$ is not isomorphic to $_\Lambda \Lambda H e_j$ and hence $\Lambda H e_i \ncong \Lambda H e_j$. Correspondingly, $\Gamma$ is a basic algebra with two indecomposable projective summands.

Now let us study the structure of $\Gamma$. Firstly, $\dim_k e_i \Lambda H e_i \geqslant |H| \geqslant p \geqslant 5$ since it contains linearly independent vectors $e_i h = h e_i$ for every $h \in H$. Furthermore,
\begin{equation*}
e_i \Lambda H e_j = \bigoplus _{h \in H} e_i \Lambda h(e_j) \otimes_k h, \quad e_j \Lambda H e_i = \bigoplus _{h \in H} e_j \Lambda e_i \otimes_k h.
\end{equation*}
Since $h$ fixes $e_i$ and sends $e_j$ to $h(e_j)$, and $e_j$ is connected to $e_i$, we deduce that either $\dim_k e_i \Lambda h(e_j) = \dim_k e_i \Lambda e_j \geqslant 1$ or $\dim_k e_j \Lambda e_i \geqslant 1$. Therefore, either $\dim_k e_i \Lambda H e_j \geqslant |H| \geqslant p \geqslant 5$ or $\dim_k e_j \Lambda H e_i \geqslant |H| \geqslant p \geqslant 5$.

In \cite{Bautista,Bongartz} the representation types of algebras with two nonisomorphic indecomposable projective modules have been classified. By the list described in these papers, we conclude that $\Gamma$ is of finite representation type only if it is isomorphic to the following algebra $R$ with relations $(\alpha \beta)^t = (\beta \alpha)^t = 0$ for some $t \geqslant p \geqslant 5$:
\begin{equation*}
\xymatrix{ e_i \ar@/^/ [rr]^{\alpha} & & e_j \ar@ /^/ [ll]^{\beta}.}
\end{equation*}
We claim that this is impossible. Indeed, if $\Gamma \cong R$, then every non-invertible element in $e_i \Gamma e_i$ is contained in $e_i \Gamma e_j \Gamma e_i = e_i \Lambda H e_j \Lambda H e_i$ since $R$ has this property. In particular, $e_i(g - 1_H) \in e_i \Gamma e_i$, which is not invertible, is contained in $e_i \Lambda H e_j \Lambda H e_i$.

We have:
\begin{align*}
e_i \Lambda H e_j \Lambda H e_i & = e_i (\bigoplus _{h \in H} \Lambda h(e_j) h) (\Lambda H e_i) \subseteq \bigoplus _{h \in H} e_i \Lambda h(e_j) h (\Lambda H e_i)\\
& \subseteq \bigoplus _{h \in H} e_i \Lambda h(e_j)  (\Lambda H e_i) = \bigoplus _{h \in H} (e_i \Lambda h(e_j) \Lambda e_i) H
\end{align*}
since $H$ fixes $e_i$. Note that $h (e_j) \neq e_i$ for every $h \in H$, so $e_i \Lambda h(e_j) \Lambda e_i \subseteq \rad(e_i \Lambda e_i)$. Thus $e_i (g - 1_H)$ is contained in $e_i \Lambda H e_j \Lambda H e_i \subseteq \rad (e_i \Lambda e_i) H$, and we can write
\begin{equation*}
e_i (g - 1_H) = \sum_{h \in H} \lambda_h h = \lambda_{1_H} 1_H + \sum_{1_H \neq h \in H} \lambda_h h
\end{equation*}
where $\lambda_h \in \rad(e_i \Lambda e_i)$. Therefore,
\begin{equation*}
(-e_i + \lambda_{1_H}) 1_H = \sum_{1_H \neq h \in H} \lambda_h h - e_i g.
\end{equation*}
This happens if and only if $e_i - \lambda_{1_H} = 0$. But this is impossible since $e_i$ is the identity element of $e_i \Lambda e_i$ and $\lambda_{1_H} \in \rad(e_i \Lambda e_i)$. Therefore, as we claimed, $\Gamma$ cannot be isomorphic to the path algebra $R$ of the above quiver, so it must be of infinite representation type.
\end{proof}

Although we have assumed the condition that $\Lambda$ has a complete set of primitive orthogonal idempotents which is closed under the action of $S$, it can be replaced by a slightly weaker condition that $\Lambda$ can be decomposed into a set of indecomposable summands which is closed under the action of $S$. With some small modifications, we can show that the above proposition is still true under this weaker condition.

In next section we will use the second statement of this proposition to classify the representation types of skew groups algebras $\Lambda G$ for $p \neq 2, 3$ when $\Lambda$ is the incidence algebra of a finite connected poset $\mathcal{P}$ and elements of $G$ act as automorphisms on $\mathcal{P}$.

We reminder the reader that the condition that $p \neq 2, 3$ is required, as shown by the following example.

\begin{example}
Let $\Lambda$ be the path algebra of the quiver $x \rightarrow y$. Let $S = \langle g \rangle$ be a cyclic group of order $p = 2$. Suppose that the action of $S$ on $\Lambda$ is trivial. That is, $g(x) = x$ and $g(y) = y$. Then $\Lambda S$ is isomorphic to the path algebra of the following quiver with relations $\alpha \delta = \theta \alpha$, $\theta^2 = \delta^2 = 0$. By Bongartz's list in \cite{Bautista}, this algebra is of finite representation type. However, $\gldim \Lambda S = \infty$ although $\gldim \Lambda = 1$.
\begin{equation*}
\xymatrix{ \mathcal{E}: & x \ar@(ld,lu)[]|{\delta} \ar[rr] ^{\alpha} & & y \ar@(rd,ru)[]|{\theta}}.
\end{equation*}

Now let $\Lambda$ be the path algebra of the following quiver such that the composite of every two arrows is 0. Let $S = \langle g \rangle$ be a cyclic group of order $p=5$. The action of $g$ on $\Lambda$ is determined by $g(i) = (i + 1)\mod 5$. Then we find that all idempotents are isomorphic in $\Lambda S$. Consequently, $\Lambda S$ is isomorphic to the algebra of $5 \times 5$ matrices over $k[X] / (X^2)$. It is Morita equivalent to $k[X] / (X^2)$, which is of finite representation type. It is also clear that both $\Lambda$ and $\Lambda S$ have infinite global dimension. Moreover, $\Lambda ^S$ is spanned by $1$ and $\alpha + \beta + \gamma + \delta + \theta$, and $\Lambda^S \cong k[X] / (X^2)$.
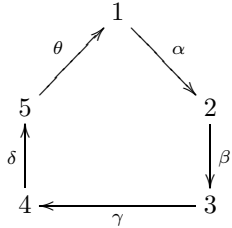
\begin{equation*}
\xymatrix{ & 1 \ar[dr]^{\alpha} & \\ 5 \ar[ur]^{\theta} & & 2 \ar[d]^{\beta} \\ 4 \ar[u]^{\delta} & & 3 \ar[ll]^{\gamma}}
\end{equation*}
\end{example}

In the rest of this section we consider in details the situation that $S$ acts freely on $E$. Then $|S|$ divides $n$, so $E$ has $s = n / |S|$ $S$-orbits. We will show that this situation generalize many results in the case that a finite group of invertible order acts on $\Lambda$ (note in this special case $S = 1$ always acts freely on $E$). In particular, the radical of $\Lambda S$ is $\mathfrak{r} S$, and a $\Lambda S$-module is projective if and only if viewed as a $\Lambda$-module it is projective.

Let $M$ be a $\Lambda S$-module, the elements $v \in M$ satisfying $g (v) = v$ for every $g \in S$ form a $\Lambda ^S$-submodule, which is denoted by $M^S$. Let $F^S$ be the functor from $\Lambda S$-mod, the category of finitely generated $\Lambda S$-modules, to $\Lambda ^S$-mod, sending $M$ to $M^S$. This is indeed a functor since for every $\Lambda S$-homomorphism $f: M \rightarrow N$, $M^S$ is mapped into $N^S$ by $f$.

Statements in the following proposition are essentially known to people under the assumption that there is some $\epsilon \in \Lambda$ such that $\Sigma _{g \in S} g(\epsilon) = 1$, see \cite{ARS, Auslander}.

\begin{proposition}
Let $S, \Lambda, E$ and $F^S$ be as before and suppose that $S$ acts freely on $E$. We have:
\begin{enumerate}
\item The $i$-th radical of $\Lambda S$ is $\mathfrak{r} ^i S$, where $\mathfrak{r}$ is the radical of $\Lambda$, $i \geqslant 1$.
\item The functor $F^S$ is exact.
\item The regular representation $_{\Lambda S} \Lambda S \cong \Lambda ^{|S|}$, where $\Lambda$ is the trivial $\Lambda S$-module.
\item The basic algebra of $\Lambda S$ is isomorphic to $\End _{\Lambda S} (\Lambda) \cong \Lambda ^S$, the space constituted of all elements in $\Lambda$ fixed by $S$.
\item A $\Lambda S$-module $M$ is projective (resp., injective) if and only if the $\Lambda$-module $_{\Lambda} M$ is projective (resp., injective). In particular, $\Lambda S$ is self-injective if and only if so is $\Lambda$.
\end{enumerate}
\end{proposition}

\begin{proof}
(1), Since $\mathfrak{r}$ is closed under the action of $S$, $\mathfrak{r} S = \mathfrak{r} \otimes_k kS$ is a two-sided ideal of $\Lambda S$. Moreover, it is nilpotent, and hence is contained in $\mathfrak{R}$, the radical of $\Lambda S$. On the other hand, $\Lambda S e_i \cong \Lambda S g(e_i)$ for $g \in S$, so the dimension of every simple $\Lambda S$-module is at least $|S|$. Furthermore, $\Lambda S / \mathfrak{R}$ has $n$ simple summands. Thus $\dim_k \Lambda S / \mathfrak{R} \geqslant n |S|$. But $\dim_k \Lambda S / \mathfrak{r} S = \dim_k (\Lambda / \mathfrak{r}) \times |S| = n |S|$. This forces $\mathfrak{R} = \mathfrak{r} S$, $\mathfrak{R} ^2 = \mathfrak {r} S \mathfrak{r} S = \mathfrak{r}^2 S$, and so on.

(2), Observe that $F$ is canonically isomorphic to $\Hom _{\Lambda S} (\Lambda, -)$ and its exactness follows from the fact that the trivial $\Lambda S$-module $\Lambda$ is projective. For details, see \cite{Auslander}.

(3), As $\Lambda S$-modules, $\Lambda S = \bigoplus _{g \in S} \Lambda S g(\epsilon) \cong (\Lambda S \epsilon) ^{|S|}$ since $\epsilon$ and $g (\epsilon)$ are isomorphic idempotents for every $g \in S$. Thus $\Lambda S$ has precisely $s = n / |S|$ non-isomorphic indecomposable projective modules $\Lambda S e_i$, $1 \leqslant i \leqslant s$. Since $\sum _{g \in S} g(\epsilon) = 1$, by Proposition 4.1 in page 87 of \cite{Auslander}, the trivial $\Lambda S$-module $_{\Lambda S} \Lambda$ is projective. It has at most $s$ indecomposable summands since for two distinct indecomposable summands $P$ and $Q$, $_{\Lambda} P$ and $_{\Lambda} Q$ are not isomorphic. Consequently, $_{\Lambda S} \Lambda$ can have at most $s$ indecomposable summands. On the other hand, it is clear that $\Lambda \sum_{g \in S} g(e_i)$ is a submodule of $_{\Lambda S} \Lambda$, and $_{\Lambda S} \Lambda = \bigoplus _{i = 1}^s (\Lambda \sum_{g \in S} g(e_i))$. Therefore, $_{\Lambda S} \Lambda$ has at least $s$ indecomposable summands. This forces $_{\Lambda S} \Lambda$ to have precisely $s$ indecomposable summands, and all of them are pairwise non-isomorphic. Consequently, $_{\Lambda S} \Lambda \cong \Lambda S \epsilon$, and
\begin{equation*}
_{\Lambda S} \Lambda S \cong (\Lambda S \epsilon) ^{|S|} \cong (_{\Lambda S} \Lambda) ^{|S|}.
\end{equation*}

(4) It follows from (3) immediately.

(5) The $\Lambda S$-module $M$ is projective if and only if the functor $\Hom _{\Lambda S} (M, -)$ is exact. But $\Hom _{\Lambda S} (M, -) = \Hom _{\Lambda} (M, -)^S$ is a composite of $\Hom_{\Lambda} (M, -)$ and the exact functor $F^S$. Therefore, $\Hom _{\Lambda S} (M, -)$ is exact if and only if $\Hom _{\Lambda} (M, -)$ is exact, which is equivalent to saying that the restricted module $_{\Lambda} M$ is projective. Using the same technique we can show that $M$ is injective if and only if so is $_{\Lambda} M$.

Suppose that $\Lambda$ is self-injective and let $P$ be a projective $\Lambda S$-module. Then $_{\Lambda} P$ is projective, so is injective. Thus $P$ is also injective, and hence $\Lambda S$ is self-injective. Conversely, suppose that $\Lambda S$ is self-injective. Take an arbitrary projective $\Lambda$-module $Q$ and consider $Q \uparrow _1^S = \Lambda S \otimes _{\Lambda} Q$. It is projective, so is injective as well. Therefore, $Q \uparrow _1^S \downarrow _1^S$ is also injective. As a summand of $Q \uparrow _1^S \downarrow _1^S$, $Q$ is also injective. Therefore, $\Lambda$ is self-injective.
\end{proof}

Note that if the action of $S$ on $E$ is free, then $\Lambda S$ is actually a matrix algebra over $\Lambda ^S$. In particular, $\Lambda S$ is a left (resp. right) free $\Lambda^S$-module. But $_{\Lambda} \Lambda S \cong (_{\Lambda} \Lambda)^{|S|}$ (resp., $\Lambda S _{\Lambda} \cong (\Lambda _{\Lambda})^{|S|}$) as left (resp., right) $\Lambda$-modules, so $_{\Lambda} \Lambda$ (resp., $\Lambda _{\Lambda}$) is also a left (resp. right) free $\Lambda^S$-module. But it might not be a free bimodule, as shown by the following example.

\begin{example} \normalfont
Let $\Lambda$ be the path algebra of the following quiver over a field $k$ with characteristic 2 satisfying the relations $\delta \beta = \beta \delta = 0$, and let $S = \langle x \rangle$ be a cyclic group of order 2, whose action on $\Lambda$ is determined by $\sigma_x (1_x) = 1_y$ and $\sigma_x (\delta) = \beta$. Then the skew group algebra $\Lambda ^{\sigma} S$ is Morita equivalent to $\Lambda^S$ spanned by $1= 1_x + 1_y$ and $\delta + \beta$, and $\Lambda^S \cong k[X] / (X^2)$.
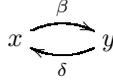
\begin{equation*}
\xymatrix{ x \ar@/^/[r] ^{\beta} & y \ar@/^/[l] ^{\delta}}
\end{equation*}

It is easy to check that $\Lambda$ is both a left free $\Lambda^S$-module and a right free $\Lambda^S$-module. Suppose that it is a free bimodule. Then $_{\Lambda ^S} \Lambda _{\Lambda ^S} \cong \Lambda^S \oplus B$ where $B$ is a free bimodule of rank 1. Since $1_x$ is not in $\Lambda^S$, we can find a unique $u \in B$ with $1_x = a + b(\beta + \delta) + u$, where $a, b \in k$. Therefore, $u = 1_x - a - b(\beta + \delta) \in B$, so $(\beta + \delta) u = (1-a) \delta - a \beta \in B$ since $B$ is a left $\Lambda^S$-module. Similarly, $u (\beta + \delta) = (1-a) \beta - a \delta \in B$ since $B$ is a right $\Lambda^S$-module. Adding these two terms we have $\beta + \delta \in B$, which is impossible. This contradiction tells us that $\Lambda$ is not a free bimodule.
\end{example}

If $\Lambda = \Lambda^S \oplus B$ as $\Lambda^S$-bimodule, we have a split bimodule homomorphism $\zeta: \Lambda \to \Lambda^S$. For $M \in \Lambda^S$-mod, we define two linear maps:
\begin{align*}
\psi: M \rightarrow (\Lambda \otimes _{\Lambda ^S} M) \downarrow _{\Lambda ^S} ^{\Lambda}, \quad v \mapsto 1 \otimes v, \quad v \in M;\\
\varphi: (\Lambda \otimes _{\Lambda ^S} M) \downarrow _{\Lambda ^S} ^{\Lambda} \rightarrow M, \quad \lambda \otimes v \mapsto \zeta(\lambda) v, \quad v \in M, \lambda \in \Lambda.
\end{align*}
These two maps are well defined $\Lambda^S$-module homomorphisms (to check it, we need the assumption that $A^S$ is a summand of $\Lambda$ as $\Lambda^S$-bimodules). We also observe that when $\lambda \in \Lambda ^S$,
\begin{equation*}
\varphi (\lambda \otimes v) = \zeta (\lambda) v = \lambda v.
\end{equation*}
Therefore, $\varphi \circ \psi$ is the identity map. Particularly, $M$ is isomorphic to a summand of $M \uparrow _{\Lambda^S} ^{\Lambda} \downarrow _{\Lambda^S} ^{\Lambda}$.

Now we are ready to prove the first main result. Recall an algebra $A$ is called an \textit{Auslander algebra} if $\gldim A \leqslant 2$, and if
\begin{equation*}
\xymatrix{0 \ar[r] & A \ar[r] & I_0 \ar[r] & I_1 \ar[r] & I_2 \ar[r] & 0}
\end{equation*}
is a minimal injective resolution, then $I_0$ and $I_1$ are projective as well.

\begin{theorem}
Let $\Lambda$ be a basic finite-dimensional $k$-algebra and $G$ be a finite group acting on $\Lambda$. Let $S \leqslant G$ be a Sylow $p$-subgroup. Suppose that there is a set of primitive orthogonal idempotents $E = \{ e_i \} _{i = 1}^n$ such that $E$ is closed under the action of $S$ and $\sum _{i=1}^n e_i = 1$. Then:
\begin{enumerate}
\item$\gldim \Lambda G < \infty$ if and only if $\gldim \Lambda < \infty$ and $S$ acts freely on $E$. Moreover, if $\gldim \Lambda G < \infty$, then $\gldim \Lambda G = \gldim \Lambda = \gldim \Lambda^S$.
\item $\Lambda G$ is an Auslander algebra if and only if so is $\Lambda$ and $S$ acts freely on $E$.
\item Suppose that $p \neq 2, 3$ and $\Lambda$ is not a local algebra. Then $\Lambda G$ is of finite representation type if and only if so is $\Lambda^S$ and $S$ acts freely on $E$. If $\Lambda \cong \Lambda^S \oplus B$ as $\Lambda^S$-bimodules, $\Lambda G$ has finite representation type if and only if so does $\Lambda$ and $S$ acts on $E$ freely.
\end{enumerate}
\end{theorem}

\begin{proof}
By Corollary 2.2, we only need to deal with the case that $G$ is a $p$-group, i.e., $G = S$.

(1): The only if part of the first statement follows from the second statement of Corollary 2.2 and Proposition 3.3. Now we prove the if part.

Suppose that $G = S$ acts freely on $E$ and $\gldim \Lambda = l < \infty$. For an arbitrary $\Lambda S$-module $M$, take a projective resolution
\begin{equation*}
\xymatrix{\ldots \ar[r] & P^1 \ar[r] & P^0 \ar[r] & M.}
\end{equation*}
Applying the restriction functor we get a projective resolution
\begin{equation*}
\xymatrix{\ldots \ar[r] & _{\Lambda} P^1 \ar[r] & _{\Lambda} P^0 \ar[r] & _{\Lambda} M.}
\end{equation*}
Therefore, for $i \geqslant 0$, $_{\Lambda} \Omega^i (M)$ is a direct sum of $\Omega^i (_{\Lambda} M)$ and a projective $\Lambda$-module. Since $\gldim \Lambda = l < \infty$, $\Omega^l (_{\Lambda} M) = 0$, so $_{\Lambda} \Omega^l (M)$ is a projective $\Lambda$-module. By (5) of Proposition 3.5, $\Omega^l (M)$ must be projective. Thus $\gldim \Lambda S \leqslant \gldim \Lambda < \infty$.

It is always true that $\gldim \Lambda S \geqslant \gldim \Lambda$. In the above proof we have actually shown that the other inequality is also true if $\gldim \Lambda < \infty$ and $S$ acts freely on $E$. Since in this situation $\Lambda S$ is Morita equivalent to $\Lambda ^S$, we have the required identity.

(2): If $\Lambda$ is an Auslander algebra, then $\gldim \Lambda \leqslant 2$. Since $G$ acts freely on $E$, we know that $\gldim \Lambda G = \gldim \Lambda \leqslant 2$ by (1). Moreover, the minimal injective resolution of $\Lambda G$-modules
\begin{equation*}
\xymatrix{0 \ar[r] & \Lambda G \ar[r] & I_0 \ar[r] & I_1 \ar[r] & I_2 \ar[r] & 0}
\end{equation*}
gives rises to
\begin{equation*}
\xymatrix{0 \ar[r] & _{\Lambda} \Lambda G \ar[r] & _{\Lambda} I_0 \ar[r] & _{\Lambda} I_1 \ar[r] & _{\Lambda} I_2 \ar[r] & 0}
\end{equation*}
which contains a minimal injective resolution of $\Lambda$ as a summand since $_{\Lambda} \Lambda G \cong \Lambda ^{|G|}$. Since $_{\Lambda} I_0$ and $_{\Lambda} I_1$ are both injective and projective, By (5) of Proposition 3.5, $I_0$ and $I_1$ are both projective and injective. Therefore, $\Lambda G$ is an Auslander algebra.

Conversely, if $\Lambda G$ is an Auslander algebra, then $\gldim \Lambda G \leqslant 2$. Therefore, $\gldim \Lambda \leqslant 2$ and $G$ must act freely on $E$ by (1). As explained in the previous paragraph, a minimal injective resolution of $\Lambda G$ gives rise to a minimal injective resolution of $\Lambda$ satisfying the required condition. Thus $\Lambda$ is an Auslander algebra.

(3): If $\Lambda S$ is of finite representation type, then $\Lambda$ has finite representation type by Corollary 2.2, and the action of $S$ on $E$ must be free by the second part of Proposition 3.3. Consequently, $\Lambda ^S$ is Morita equivalent to $\Lambda S$ by (4) of Proposition 3.5, and is of finite representation type.

Conversely, if $S$ acts freely on $E$, again $\Lambda S$ is Morita equivalent to $\Lambda ^S$, and its finite representation type implies the finite representation type of $\Lambda S$. If furthermore $\Lambda \cong \Lambda^S \oplus B$ as $\Lambda^S$-bimodules, then every indecomposable $\Lambda^S$-module $M$ is isomorphic to a summand of $M \uparrow _{\Lambda^S} ^{\Lambda} \downarrow _{\Lambda^S} ^{\Lambda}$. That is, $M$ is isomorphic to a summand of $V \downarrow _{\Lambda^S} ^{\Lambda}$ for a certain indecomposable $\Lambda$-module $V$. If $\Lambda$ is of finite representation type, so are $\Lambda^S$ and $\Lambda S$.
\end{proof}

\section{Representation types of transporter categories}

 Throughout this section let $\mathcal{P}$ be a connected finite poset with a partial order $\leqslant$ and $G$ be a finite group of automorphisms on $\mathcal{P}$. We consider the \textit{transporter category} $\mathcal{T} = G \propto \mathcal{P}$. The category $\mathcal{T}$ has the same objects as $\mathcal{P}$ (viewed as a finite category). All morphisms in $\mathcal{T}$ are of the form $\alpha g$, where $g \in G$ and $\alpha: x \rightarrow y$ is a morphism in $\mathcal{P}$ (which means $x \leqslant y$ in $\mathcal{P}$). Note that $\alpha g$ has source $g^{-1} (x)$ and target $y$. For objects $x, y \in \Ob \mathcal{T}$, we denote the set of morphisms from $x$ to $y$ by $\mathcal{T} (x, y)$, and in particular, we let $G_x = \mathcal{T} (x, x)$. Transporter categories are examples of \textit{finite EI categories}, i.e, small categories with finitely many morphisms such that every endomorphism is an isomorphism. Representations of these structures have been studied in \cite{Li1,Li2,Webb1,Webb2,Xu1,Xu2}.

Statements in the following proposition are immediate results of the definition, see \cite{Xu2}.

\begin{proposition}
Let $G, \mathcal{P}$ and $\mathcal{T}$ be as above. Then:
\begin{enumerate}
\item for each object $x$, $G_x = \{ g \in G \mid g(x) = x \}$;
\item for objects $x, y$ such that $\mathcal{T} (x, y) \neq \emptyset$, both $G_x$ and $G_y$ act freely on $\mathcal{T} (x, y)$;
\item two objects $x, y \in \Ob \mathcal{T}$ are isomorphic if and only if they are in the same $G$-orbit.
\end{enumerate}
\end{proposition}

We focus on the representation type of $\mathcal{T}$. By definition, a \textit{representation} of $\mathcal{T}$ is a covariant functor from $\mathcal{T}$ to $k$-vec, the category of finite-dimensional vector spaces. Since the category of representations of $\mathcal{T}$ is equivalent to the category of finitely generated $k\mathcal{T}$-modules, where $k\mathcal{T}$ is the \textit{category algebra} of $\mathcal{T}$ \cite{Webb1}, we identify in this paper representations of $\mathcal{T}$ and $k\mathcal{T}$-modules. In \cite{Xu2} Xu shows that the category algebra $k\mathcal{T} \cong k\mathcal{P} G$, the skew group algebra of the incidence algebra $k\mathcal{P}$. Clearly, the identity morphisms $\{ 1_x \} _{x \in \Ob \mathcal{T}}$ form a complete set of primitive orthogonal idempotent of $k \mathcal{P}$ which is closed under the action of $G$.

We first point out that the representation type of $\mathcal{T}$ is not completely determined by the representation types of $kG$ and $\mathcal{P}$. Actually, the action of $G$ on $\mathcal{P}$ plays an important role, as shown by the following example.

\begin{example}
Let $\mathcal{P}$ be the following poset and $G = \langle g \rangle$ be a cyclic group of order $p = 3$.
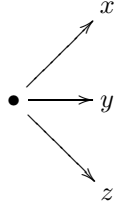
\begin{equation*}
\xymatrix{ & x \\ \bullet \ar[ur] \ar[r] \ar[dr] & y \\ & z}
\end{equation*}

If $G$ acts trivially on $\mathcal{P}$, then $k\mathcal{T} \cong k\mathcal{P} [X] / (X^3]$. Using covering theory described in \cite{Bongartz} we can show that it is of infinite representation type.

On the other hand, there is another action such that $g(x) = y$ and $g(y) = z$. In this case the category algebra of the skeletal subcategory $\mathcal{E}$ of $\mathcal{T}$ is isomorphic to the path algebra of the following quiver with relation $\delta^3 = 0$:
\begin{equation*}
\xymatrix{ \mathcal{E}: & x \ar@(ld,lu)[]|{\delta} \ar[r] & y}.
\end{equation*}
Since this path algebra is of finite representation type, so is $k\mathcal{T}$.
\end{example}

From this example we find that the finite representation types of $G$ and $\mathcal{P}$ do not guarantee the finite representation type of $\mathcal{T}$. Conversely, suppose that $\mathcal{T}$ is of finite representation type. Although by Corollary 2.2, $\mathcal{P}$ must be of finite representation type as well (finite connected posets of finite representation type have been classified by Loupias in \cite{Loupias}), $G$ may be of infinite representation type, as described in the following example.

\begin{example}
Let $\mathcal{P}$ be the poset of 4 incomparable elements and $G$ be the symmetric group over 4 letters permuting these 4 incomparable elements. Let $k$ be an algebraically closed field of characteristic 2. Then $kG$ is of infinite representation type. However, the four objects in the transporter category $\mathcal{T}$ are all isomorphic and the automorphism group of each object is a symmetric group $H$ over 3 letters. Therefore, the category algebra $k \mathcal{T}$ is Morita equivalent to $kH$, which is of finite representation type.
\end{example}

Now we can show that $\mathcal{T}$ is of finite representation type only for very few cases when $p \neq 2, 3$.

\begin{proposition}
Let $G$ be a finite group acting on a connected finite poset $\mathcal{P}$ and suppose that $p \neq 2, 3$. Let $S$ be a Sylow $p$-subgroup of $G$. If the transporter category $\mathcal{T} = G \propto \mathcal{P}$ is of finite representation type, then one of the following conditions must be true:
\begin{enumerate}
\item $|G|$ is invertible in $k$ and $\mathcal{P}$ is of finite representation type;
\item $\mathcal{P}$ has only one element and $G$ is of finite representation type;
\item if $S \neq 1$, then the skeletal category $\mathcal{E}$ of the transporter category $\mathcal{S} = S \propto \mathcal{P}$ is a poset of finite representation type.
\end{enumerate}
\end{proposition}

\begin{proof}
Suppose that $\mathcal{T}$ is of finite representation type, or equivalently, $k\mathcal{T}$ is of finite representation type. If $|G|$ is invertible, we immediately get (1). If $\mathcal{P}$ has only one element, (2) must be true. Therefore, assume that $|G|$ is not invertible and $\mathcal{P}$ is nontrivial. Thus $S \neq 1$ and $\mathcal{E}$ is a connected skeletal finite EI category with more than one objects. By Corollary 2.2, $\mathcal{S}$ and hence $\mathcal{E}$ are of finite representation type.

Note that $k\mathcal{P}$ is not a local algebra. By Proposition 3.3, $S$ must act freely on $\mathcal{P}$, so $\mathcal{S} (x, x) = 1$ for every $x \in \mathcal{S}$. Consequently, $\mathcal{E} (x,x) = 1$ for every $x \in \Ob \mathcal{E}$, and $\mathcal{E}$ is an acyclic category, i.e., a skeletal finite EI category such that the automorphism group of every object is trivial. Since $\mathcal{E}$ is of finite representation type, we conclude that $|\mathcal{E} (x, y)| \leqslant 1$ for all $x, y \in \Ob \mathcal{E}$. Otherwise, we can find some $x, y \in \Ob \mathcal{E}$ such that there are at least two morphisms from $x$ to $y$. Considering the full subcategory of $\mathcal{E}$ formed by $x$ and $y$ we get a Kronecker quiver, which is of infinite representation type. This is impossible.

We have shown that $\mathcal{E}$ is a skeletal finite EI category such that the automorphism groups of all objects are trivial and the number of morphisms between any two distinct objects cannot exceed 1. It is indeed a poset. Clearly, it must be of finite representation type.
\end{proof}

Therefore, for $p \neq 2,3$ and nontrivial $\mathcal{P}$, if $\mathcal{T}$ is of finite representation type, then all Sylow $p$-subgroups of $G$ must act freely on the poset $\mathcal{P}$. In the rest of this paper we will show that third case cannot happen. That is, if a nontrivial Sylow $p$-subgroup $S \leqslant G$ acts freely on the poset, then $\mathcal{T}$ must be of infinite representation type.

First we describe a direct corollary of this proposition.

\begin{corollary}
Let $G$ be a finite group having a Sylow $p$-subgroup acting nontrivially on a connected finite poset $\mathcal{P}$. If $p \geqslant 5$ and $\mathcal{P}$ is a directed tree, then the transporter category $\mathcal{T} = G \propto \mathcal{P}$ is of infinite representation type.
\end{corollary}

\begin{proof}
Without loss of generality we assume that $G$ is a $p$-group. Since the action of $G$ is nontrivial, $\mathcal{P}$ must be nontrivial as well. Therefore, if $\mathcal{T}$ is of finite representation type, $G$ must act freely on $\mathcal{P}$. Suppose that $\mathcal{P}$ has $n$ elements, $n \geqslant 2$. Then it has $n-1$ arrows. But $G$ acts freely both on the set of elements and on the set of arrows. Therefore, $|G|$ divides both $n$ and $n-1$. This forces $|G| = 1$, which is impossible.
\end{proof}

The following lemma is useful to determine the representation types of $\mathcal{T}$ for some cases.

\begin{lemma}
Let $G$ be a finite group acting freely on a connected finite poset $\mathcal{P}$. If there exist distinct elements $x, y, z \in \mathcal{P}$ such that $x < y$ and $x < z$ (or $x > y $ and $x > z$) and $y, z$ are in the same $G$-orbit, then $\mathcal{T} = G \propto \mathcal{P}$ is of infinite representation type.
\end{lemma}

\begin{proof}
Note that $G$ is nontrivial since one of its orbits contains at least two elements. Without loss of generality we assume that there are elements $x, y, z \in \mathcal{P}$ such that $x < y$ and $x < z$ where $y$ and $z$ are in the same $G$-orbit. Then we can find some $1_G \neq g \in G$ with $g(z) = y$. Let $\alpha$ and $\beta$ be the morphisms $x < y$ and $x < z$ in $\mathcal{P}$ respectively.

Consider the skeletal category $\mathcal{E}$ of $\mathcal{T}$. Since $G$ acts freely on $\mathcal{P}$, $G_u = \mathcal{E} (u, u) = 1$ for every $u \in \Ob \mathcal{P}$. Note that elements in $\mathcal{P}$ lie in disjoint $G$-orbits, and the objects of $\mathcal{E}$ form a set of representatives of these $G$-orbits. Clearly, we can let $x, y \in \Ob \mathcal{E}$ since they are in distinct $G$-orbits.

We claim that there are more than one morphisms from $x$ to $y$ in $\mathcal{E}$. If this is true, then $\mathcal{E}$ must be of infinite representation type since $\mathcal{E} (x, x) \cong 1 \cong \mathcal{E} (y, y)$, and the conclusion follows.

Clearly, $\alpha 1_G$ is a morphism in $\mathcal{E} (x, y)$. Another morphism in this set is $g (\beta) g$. Indeed, let $1_y1_G$ and $1_x1_G$ be the identity morphism of $y$ and $x$ in $\mathcal{S},$ we check:
\begin{align*}
1_y 1_G \cdot g (\beta) g = 1_y g(\beta) g = g(\beta) g
\end{align*}
since $g(\beta)$ has target $y$; and
\begin{align*}
g(\beta) g \cdot 1_x 1_G = g (\beta 1_x) g = g(\beta) g
\end{align*}
since $\beta$ has source $x$. Thus there are indeed at least two morphisms $\alpha 1_G$ and $g(\beta) g$ in $\mathcal{E} (x, y)$.
\end{proof}

\begin{example}
Let $\mathcal{P}$ be the following poset and $G = \langle g \rangle$ be a cyclic group of order $p = 3$. The action of $G$ on $\mathcal{P}$ is defined by $g (n) = (n + 2) \mod 6$, $1 \leqslant n \leqslant 6$.
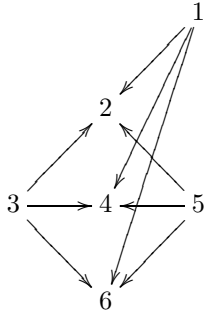
\begin{equation*}
\xymatrix{ & & 1 \ar[ld] \ar[ldd] \ar[lddd] \\
 & 2 &\\
3 \ar[ru] \ar[r] \ar[rd] & 4 & 5 \ar[lu] \ar[l] \ar[ld]\\
 & 6 &}
\end{equation*}

Since $G$ acts freely on $\mathcal{P}$, we get $G_n = 1$ for $1 \leqslant n \leqslant 6$. But for every element in this poset we can find three arrows starting or ending at elements in the same $G$-orbit. Therefore, the transporter category $G \propto \mathcal{P}$ is of infinite representation type.
\end{example}

The next lemma is used in the proof of the main result of this section.

\begin{lemma}
Let $\mathcal{P}$ be a connected finite poset and $G$ be a nontrivial finite group acting on it freely. Then there exists a non-oriented cycle $\mathcal{C}$ in $\mathcal{P}$ such that $\mathcal{C}$ is closed under the action of some $1_G \neq g \in G$, and every element in $\mathcal{C}$ (viewed as a poset) is either maximal or minimal.
\end{lemma}

\begin{proof}
For every element $x \in \mathcal{P}$ we let $Gx$ be the $G$-orbit where $x$ lies. Since $\mathcal{P}$ is connected, for every $x \in \mathcal{P}$ and each $x \neq y \in Gx$ we can find a non-oriented path connecting $x$ and $y$. Let $\gamma$ be a non-oriented path satisfying the following conditions: the two endpoints $x$ and $y$ of $\gamma$ are distinct and lie in the same $G$-orbit; $\gamma$ has the minimal length. This $\gamma$ exists (although might not be unique) since $\mathcal{P}$ is connected.

Note that the length of $\gamma$ is at least 2 since its two endpoints are in the same $G$-orbit and they are incomparable in $\mathcal{P}$. Moreover, distinct elements in $\gamma$ lie in different $G$-orbits except the two endpoints. Indeed, if there are two vertices $u$ and $v$ in $\gamma$ lying in the same $G$-orbit and $\{u, v\} \neq \{x, y\}$, then we can take a proper segment of $\gamma$ with endpoints $u$ and $v$ satisfying the required conditions. This contradicts the assumption that $\gamma$ is shortest among all paths satisfying the required conditions.

By the construction of $\gamma$, we can find a unique $1_G \neq g \in G$ such that $y = g(x)$. Note that $n = |g| > 1$. For $0 \leqslant i \leqslant n-1$, let $g^i(\gamma)$ be the corresponding path (which is isomorphic to $\gamma$ as posets) starting at $g^i(x)$ and ending at $g^i(y)$. Since $g^i (y) = g^{i+1} (x)$, we can glue these paths to get a loop $\tilde {\mathcal{C}}$. Since $G$ acts freely on $\mathcal{P}$, and elements in each $g^i (\gamma)$ (except the two endpoints) are in distinct $G$-orbits, we deduce that $\tilde {\mathcal{C}}$ has no self-intersection. Therefore, it is a non-oriented cycle, and is closed under the action of $g$. Now we pick up all maximal elements and minimal elements in $\tilde{\mathcal{C}}$ and get another non-oriented cycle $\mathcal{C}$, which is still closed under the action of $g$. This cycle $\mathcal{C}$ is what we want.
\end{proof}

The following example explains our construction.

\begin{example}
Let $\mathcal{P}, G$ and $p$ be as in Example 4.7. Then $\gamma: 1 \rightarrow 2 \leftarrow 3$ is a path connecting $1$ and $3$. The cycle $\mathcal{C}$ is as below:
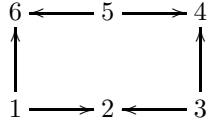
\begin{equation*}
\xymatrix{6 & 5 \ar[r] \ar[l] & 4 \\ 1 \ar[u] \ar[r] & 2 & 3 \ar[u] \ar[l]}
\end{equation*}
It is closed under the action of $G$.
\end{example}

Now we can show that the third case in Proposition 4.4 cannot happen.

\begin{proposition}
Let $G$ be a nontrivial finite group acting freely on a connected finite poset $\mathcal{P}$. Then the transporter category $\mathcal{T} = G \propto \mathcal{P}$ is of infinite representation type.
\end{proposition}

\begin{proof}
We have two cases.\\

\textbf{Case I:} There are distinct $x, y, z \in \mathcal{P}$ such that either $x < y$ and $x < z$ or $x > y$ and $x > z$, and $y$ and $z$ are in the same $G$-orbit. Then by Lemma 4.6, $\mathcal{T}$ is of infinite representation type.\\

\textbf{Case II:} The situation described in Case I does not happen. By Lemma 4.8, there are non-oriented cycles $\mathcal{C}$ such that every element in $\mathcal{C}$ is either maximal or minimal, and $\mathcal{C}$ is closed under the action of some $1_G \neq g \in G$. Take such a cycle with minimal length and still denote it by $\mathcal{C}$. We claim that $\mathcal{C}$ is a full subcategory of $\mathcal{P}$. To establish this claim, it suffices to show that two incomparable elements in $\mathcal{C}$ are still incomparable in $\mathcal{P}$.

Let $x$ and $y$ be two incomparable elements in $\mathcal{C}$ and suppose that $x < y$ in $\mathcal{P}$. Then $x$ and $y$ must be in distinct $G$-orbits, and hence in distinct $H$-orbits, where $H = \langle g \rangle$. Moreover, $y$ cannot be adjacent to some element $x' \in  Hx$ in $\mathcal{C}$. Otherwise, we have two possible choices as follows,
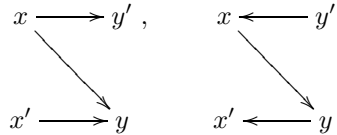
\begin{equation*}
\xymatrix{x \ar[r] \ar[dr] & y' \\ x' \ar[r] & y}, \qquad \xymatrix{x \ar[dr] & y' \ar[l] \\ x' & y \ar[l],}
\end{equation*}
where $y$ and $y'$ are in the same $H$-orbit. Note that $x \neq x'$ and $y \neq y'$ since $x$ and $y$ have been assumed to be incomparable in $\mathcal{C}$. But in the first choice we get the situation described Case I, and in the second situation $y$ is neither maximal nor minimal in $\mathcal{C}$. Therefore, both situations cannot happen.

Now let $\gamma$ be a shortest path in $\mathcal{C}$ connecting $x$ and some element in $x \neq z \in Hx$. Then vertices of $\gamma$, except $x$ and $z$, are contained in distinct $H$-orbits. Consider the following picture:
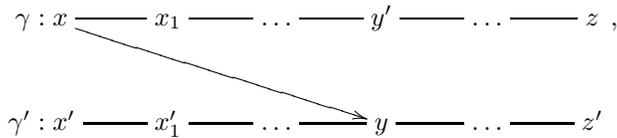
\begin{equation*}
\xymatrix{\gamma: x \ar@{-}[r] \ar[drrr] & x_1 \ar@{-}[r] & \ldots \ar@{-}[r] & y' \ar@{-}[r] & \ldots \ar@{-}[r] & z\\
\gamma': x' \ar@{-}[r] & x_1' \ar@{-}[r] & \ldots \ar@{-}[r] & y \ar@{-}[r] & \ldots \ar@{-}[r] & z'},
\end{equation*}
where $\gamma$ and $\gamma'$ coincide if $y$ is contained in $\gamma$. Note that $y$ is distinct from $x, x', z, z'$ since $x, x', z, z'$ are all in the same $H$-orbit. Moreover, $y$ is not adjacent to $x, x', z, z'$ by the reasoning in the previous paragraph. Therefore, we get a non-oriented path:
\begin{equation*}
\xymatrix{\theta: x \ar[r] & y \ar@{-}[r] & \ldots \ar@{-}[r] & z'}.
\end{equation*}
Clearly, the length $l(\theta) < l(\gamma)$.

Now as we did in the proof of Lemma 4.8, by using $\theta$ and the action of $H$ on it we obtain a non-oriented cycle $\tilde {\mathcal{D}}$. Picking up maximal and minimal elements in $\tilde {\mathcal{D}}$ (viewed as a poset) we obtain another cycle $\mathcal{D}$ satisfying that it is closed under the action of $H$ and every element in $\mathcal{D}$ (viewed as a poset) is either maximal or minimal. But $\mathcal{D}$ is clearly shorter than $\mathcal{C}$. This contradicts our choice of $\mathcal{C}$. Therefore, $x$ and $y$ are incomparable in $\mathcal{P}$ as well, and our claim is proved. That is, $\mathcal{C}$ is a full subcategory of $\mathcal{P}$.

Since both the path $\gamma$ and the group $H$ is nontrivial, $\mathcal{C}$ has at least 4 objects. Moreover, every element in $\mathcal{C}$ is either maximal or minimal. Therefore, the structure of $\mathcal{C}$ is as below:
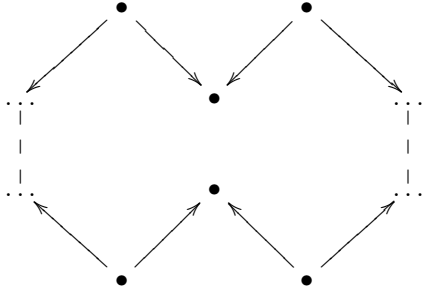
\begin{equation*}
\xymatrix{ & \bullet \ar[dr] \ar[dl] & & \bullet \ar[dr] \ar[dl] & \\
\ldots \ar@{--}[d] & & \bullet & & \ldots \ar@{--}[d] \\
\ldots & & \bullet & & \ldots\\
 & \bullet \ar[ul] \ar[ur] & & \bullet \ar[ul] \ar[ur] & }.
\end{equation*}
It is of infinite representation type either by Loupias's classification (viewing $\mathcal{C}$ as a poset) \cite{Loupias} or by Gabriel's theorem (viewing $\mathcal{C}$ as a quiver). Therefore, $\mathcal{P}$ is also of infinite representation type. By of Corollary 2.2, $\mathcal{T} = G \propto \mathcal{P}$ is of infinite representation type as well.
\end{proof}

This proposition is true for all characteristics and all finite groups. Note that in Example 4.9 the cycle $\mathcal{C}$ is not a full subcategory of $\mathcal{P}$ since the situation in Case I indeed happens. Moreover, the condition that $\mathcal{P}$ is connected is required. Consider the following example:

\begin{example}
Let $\mathcal{P}$ be the following poset and $G = \langle g \rangle$ be a cyclic group of order $p=2$, where $g(i) = i'$, $i = 1, 2$. Then $G$ acts freely on $\mathcal{P}$. But the skeletal category of $\mathcal{T} = G \propto \mathcal{P}$ is isomorphic to a component of $\mathcal{P}$ (viewed as a category), so is of finite representation type.
\begin{equation*}
\xymatrix{1 \ar[r] & 2 & 1' \ar[r] & 2'}.
\end{equation*}
\end{example}

We are ready to classify the representation types of transporter categories for $p \geqslant 5$.

\begin{theorem}
Let $G$ be a finite group acting on a connected finite poset $\mathcal{P}$ and suppose that $p \neq 2, 3$. Then the transporter category $\mathcal{T} = G \propto \mathcal{P}$ is of finite representation type if and only if one of the following conditions is true:
\begin{enumerate}
\item $|G|$ is invertible in $k$ and $\mathcal{P}$ is of finite representation type;
\item $\mathcal{P}$ has only one element and $G$ is of finite representation type.
\end{enumerate}
\end{theorem}

\begin{proof}
It suffices to show that the third case in Proposition 4.4 cannot happen. Let $S \neq 1$ be a Sylow $p$-subgroup of $G$ and suppose that $S$ acts freely on $\mathcal{P}$. Then $\mathcal{S} = S \propto \mathcal{P}$ is of infinite representation type by the previous proposition, so is $\mathcal{T}$ by Corollary 2.2.
\end{proof}

The conclusion is not true for $p = 2$ or $3$, see Example 4.2. For $p = 2$ or 3, in practice we can construct the covering of the skeletal category of $\mathcal{T}$ and use it to examine the representation type. For details, see \cite{Bongartz}.

\section{Koszul Properties of Skew Group Algebras}

In this section we study the Koszul properties of skew group algebras. Let $\Lambda = \bigoplus _{i \geqslant 0} \Lambda_i$ be a locally finite graded $k$-algebra generated in degrees 0 and 1, i.e., $\Lambda_i \Lambda_j = \Lambda_{i + j}$ and $\dim_k \Lambda_i < \infty$ for $i, j \geqslant 0$. Let $G$ be a finite group of grade-preserving algebra automorphisms. That is, for every $i \geqslant 0$ and $g \in G$, $g (\Lambda_i) = \Lambda_i$. Then the skew group algebra $\Lambda G$ is still a locally finite graded algebra generated in degrees 0 and 1 with $(\Lambda G)_0 = \Lambda_0 \otimes kG$. Since $(\Lambda G)_0$ in general is not semisimple, the classical Koszul theory describe in \cite{BGS, Mazorchuk} cannot apply, and we have to rely on generalized Koszul theories not requiring the semisimple property of $(\Lambda G)_0$, for example \cite{Green,Li3,Li4,Madsen,Woodcock}. In this paper we use the generalized Koszul theory developed in \cite{Li3,Li4}.

First we describe some definitions and results of this theory. For proofs and other details, please refer to \cite{Li3,Li4}. Let $A = \bigoplus _{i \geqslant 0} A_i$ be a locally finite graded algebra generated in degrees 0 and 1, and $M$ be a locally finite graded $A$-module generated in degree 0.

\begin{definition}
We call $M$ a generalized Koszul module if it has a linear projective resolution
\begin{equation*}
\xymatrix{ \ldots \ar[r] & P^{n} \ar[r] & \ldots \ar[r] & P^1 \ar[r] & P^0 \ar[r] & M }
\end{equation*}
such that each $P^i$ is generated in degree $i$. The algebra $A$ is said to be a generalized Koszul algebra if $A_0 \cong A /J$ is a generalized Koszul $A$-module, where $J = \bigoplus _{i \geqslant 1} A_i$.
\end{definition}

It is clear that from this definition $M$ is generalized Koszul if and only if for every $i \geqslant 0$, its $i$-th syzygy $\Omega ^i (M)$ is generated in degree $i$. For a generalized Koszul module $M$, since $\Omega ^{i+1} (M)$ is generated in degree $i+1$, $\Omega^i (M)_i \cong P^i_i$ is a projective $A_0$-module.

Note that $\Ext _A^{\ast} (M, A_0)$ as an $E = \Ext _A^{\ast} (A_0, A_0)$-module has a natural grading. If $A_0$ is semisimple, it is well known that $M$ is a Koszul $A$-module if and only if $\Ext _A^{\ast} (M, A_0)$ as an $E$-module is generated in degree 0. This is not true for general $A_0$. However, we have the following result:

\begin{theorem}
(Theorem 2.16, \cite{Li3}.) Let $A, M$ and $E$ as before. Then $M$ is a generalized Koszul $A$-module if and only if $\Ext _A^{\ast} (M, A_0)$ as an $E$-module is generated in degree 0, and $\Omega ^i (M)_i$ is a projective $A_0$-module for every $i \geqslant 0$.
\end{theorem}

We call $M$ a \textit{quasi-Koszul} $A$-module if $\Ext^{\ast}_A (M, A_0)$ as an $E$-module is generated in degrees 0. Correspondingly, $A$ is a \textit{quasi-Koszul} algebra if $E$ as a graded algebra is generated in degrees 0 and 1.

This generalized Koszul theory and the classical theory are closely related by the following result. Let $\mathfrak{r} = \rad A_0$ and $\mathfrak{R} = A \mathfrak{r} A$ be the two-sided graded ideal. Define $\bar{A} = A/ \mathfrak{R}$ and $\bar{M} = M / \mathfrak{R} M$. It is clear that $\bar{A}$ is a graded algebra and $\bar{M}$ is a graded $\bar{A}$-module.

\begin{theorem}
(Theorem 3.12, \cite{Li4}.) Suppose that both $M$ and $A$ are projective $A_0$-modules, and that $\mathfrak{r} A_1 = A_1 \mathfrak{r}$. Then $M$ is a generalized Koszul $A$-module if and only if $\bar{M}$ is a classical $\bar{A}$-module. In particular, $A$ is a generalized Koszul algebra if and only if $\bar{A}$ is a classical Koszul algebra.
\end{theorem}

The graded algebra $A$ is said to have the splitting property $(S)$ if every exact sequence $0 \rightarrow P \rightarrow Q \rightarrow R \rightarrow 0$ of left (resp., right) $A_0$-modules splits whenever $P$ and $Q$ are left (resp., right) projective $\Lambda_0$-modules. If $A_0$ is self-injective or is a direct sum of local algebras, the splitting property is satisfied.

The classical Koszul duality can be generalized as follows:

\begin{theorem}
(Theorem 4.1, \cite{Li3}.) Suppose that $A$ has the splitting property. If $A$ is a generalized Koszul algebra, then $F = \Ext ^{\ast}_A (-, A_0)$ gives a duality between the category of generalized Koszul $A$-modules and the category of generalized Koszul $E$-modules. That is, if $M$ is a generalized Koszul $A$-module, then $F(M)$ is a generalized Koszul $E$-module, and $F_E (F(M)) = \Ext ^{\ast} _E (F(M), E_0) \cong M$ as graded $A$-modules.
\end{theorem}

Now we consider the generalized Koszul properties of skew group algebra $\Lambda G$. Firstly, we observe that for a graded $\Lambda G$-module $M$ and for every $i \geqslant 0$, $M_i$ is not only a $\Lambda_0$-module, but a $kG$-module.

\begin{lemma}
Let $M$ be a graded $\Lambda G$-module whose restricted module $_{\Lambda} M$ is a generalized Koszul $\Lambda$-module, and let
\begin{equation*}
\xymatrix{ \ldots \ar[r] & P^{n} \ar[r] ^{d_n} & \ldots \ar[r] & P^1 \ar[r] ^{d_1} & P^0 \ar[r] ^{d_0} & _{\Lambda} M}
\end{equation*}
be a minimal graded projective resolution. Then every $P^i$ is a finitely generated $\Lambda G$-module.
\end{lemma}

\begin{proof}
Let $K^{i+1}$ be the kernel of $d_i$ for $i \geqslant 0$. Since $_{\Lambda} M$ is generalized Koszul, $P^i$ is generated in degree $i$. In particular, $P^0_0 \cong M_0$ is a $kG$-module. The action of $g \in G$ on $\Lambda \otimes _{\Lambda_0} M_0$ is defined by $g (\lambda \otimes v) = g(\lambda) \otimes g(v)$. Therefore, $P^0 \cong \Lambda \otimes _{\Lambda_0} P_0^0$ is a $kG$-module, and hence a $\Lambda G$-module.

The above resolution induces short exact sequences
\begin{equation*}
\xymatrix{0 \ar[r] & K_j^1 \ar[r] & P^0_j \ar[r] & M_j \ar[r] & 0}
\end{equation*}
for $j \geqslant 0$. Thus each $K_j^1$ is also a $kG$-module, in particular so is $K^1_1$. Therefore, $P^1 \cong \Lambda \otimes _{\Lambda_0} K^1_1$ is a $\Lambda G$-module. By induction, every $P^i$ is a $\Lambda G$-module.

Note that the category of locally finite $\Lambda$-modules is invariant under taking syzygies. Therefore, each $P^i$ is locally finite. Since it is generated in degree $i$ as a $\Lambda$-module, it is a finitely generated $\Lambda$-module. Clearly, viewed as a $\Lambda G$-module, it is also finitely generated.
\end{proof}

We state the following easy fact as a lemma.

\begin{lemma}
Let $P$ be a locally finite graded $\Lambda G$-module. If viewed as an $\Lambda$-module it is graded projective and generated in degree $i$, then $P \otimes_k kG$ is a graded projective $\Lambda G$-module generated in degree $i$.
\end{lemma}

\begin{proof}
Clearly $P \otimes_k kG$ is a graded $\Lambda$-module generated in degree $i$, on which $\Lambda$ acts on the left side and $G$ acts diagonally. Since $P$ is isomorphic to a summand of $\Lambda ^n$ for some $n \geqslant 0$, $P \otimes _k kG$ is isomorphic to a summand of $\Lambda ^n \otimes_k kG \cong (\Lambda G) ^n$, so is a projective $\Lambda G$-module.
\end{proof}

We show that $\Lambda$ and $\Lambda G$ share the common property of being generalized Koszul algebras. This generalizes a result of Martinez in \cite{Martinez}, where both $\Lambda_0$ and $kG$ are supposed to be semisimple. \footnote{In a note sent to the author, applying the Koszul complex Witherspoon proved that if $\Lambda$ is a Koszul algebra and $\Lambda_0 \cong k$, then $\Lambda G$ is a generalized Koszul over $kG$. Her proof actually works for the case that $\Lambda_0$ is self-injective since under this condition the generalized Koszul complex is defined. For details, see \cite{Li3}.}

\begin{theorem}
Let $M$ be a graded $\Lambda G$-module. Then the $\Lambda G$-module $M \otimes_k kG$ is a generalized Koszul $\Lambda G$-module if and only if viewed as a $\Lambda$-module $M$ is generalized Koszul. In particular, $\Lambda G$ is a generalized Koszul algebra over $(\Lambda G)_0 = \Lambda_0 \otimes_k kG$ if and only if $\Lambda$ is a generalized Koszul algebra over $\Lambda_0$.
\end{theorem}

\begin{proof}
Since $M$ is a generalized Koszul $\Lambda$-module, it has a minimal linear projective resolution
\begin{equation*}
\xymatrix{ \ldots \ar[r] & P^{n} \ar[r] & \ldots \ar[r] & P^1 \ar[r] & P^0 \ar[r] & M},
\end{equation*}
where each $P_i$ is a graded projective $\Lambda$-module generated in degree $i$.

Applying $- \otimes_k kG$, we get an exact sequence
\begin{equation*}
\xymatrix{ \ldots \ar[r] & P^n \otimes_k kG \ar[r] & \ldots \ar[r] & P^0 \otimes_k kG \ar[r] & M \otimes_k kG \ar[r] & 0}.
\end{equation*}
By the previous lemmas, each $P^i \otimes_k kG$ is actually a graded projective $\Lambda G$-module generated in degree $i$. Consequently, this sequence gives rise to linear projective resolution of $M \otimes_k kG$. By definition, $M \otimes_k kG$ is a generalized Koszul $\Lambda G$-module.

Conversely, suppose that $M \otimes_k kG$ is a generalized Koszul $\Lambda G$-module. It has a linear projective resolution
\begin{equation*}
\xymatrix{ \ldots \ar[r] & \tilde{P}^{n} \ar[r] & \ldots \ar[r] & \tilde{P}^1 \ar[r] & \tilde{P}^0 \ar[r] & M \otimes_k kG}.
\end{equation*}
Viewed as $\Lambda$-modules, we get a resolution
\begin{equation*}
\xymatrix{ \ldots \ar[r] & _{\Lambda} \tilde{P}^{n} \ar[r] & \ldots \ar[r] & _{\Lambda} \tilde{P}^1 \ar[r] & _{\Lambda} \tilde{P}^0 \ar[r] & _{\Lambda} (M \otimes _k kG) \ar[r] \ar[r] & 0}.
\end{equation*}
Clearly, every $_{\Lambda} \tilde{P}^i$ is a projective $\Lambda$-module generated in degree $i$. Moreover, $_{\Lambda} (M \otimes_k kG) \cong (_{\Lambda} M) ^{|G|}$. Therefore, $_{\Lambda} M$ has a linear projective resolution, so is a generalized Koszul $\Lambda$-module.

In particular, take $M = \Lambda_0$, which is a $\Lambda G$-module, we deduce that $\Lambda_0 \otimes_k kG$ is a generalized Koszul $\Lambda G$-module if and only if $\Lambda_0$ is a generalized Koszul $\Lambda$-module. That is, $\Lambda G$ is a generalized Koszul algebra if and only if so is $\Lambda$.
\end{proof}

Our next goal is to show that the extension algebra $\Ext _{\Lambda G} ^{\ast} ((\Lambda G)_0, (\Lambda G)_0)$ is a skew group algebra of $\Ext _{\Lambda} ^{\ast} (\Lambda_0, \Lambda_0)$. This result has been proved by Martinez in \cite{Martinez} under the assumption that both $\Lambda_0$ and $kG$ are semismple. With a careful check, we find that this result still holds for the general situation. In the rest of this section we state some main results, and briefly explain why these results still hold in our framework. We claim no originality for these results since his arguments actually work with little modifications. The reader is suggested to refer to \cite{Martinez} for details.

Let $M$ and $N$ be two finitely generated $\Lambda G$-modules. Then $\Hom _{\Lambda} (M, N)$ has a natural $\Lambda G$-module structure by letting $(g \ast f) (v) = gf(g^{-1} v)$ for every $g \in G$ and $v \in M$. It is also well known that $\Hom _{\Lambda G} (M, N) = \Hom _{\Lambda} (M, N)^G$, the set of all $\Lambda$-module homomorphisms fixed by the action of every $g \in G$.

For a fixed $g \in G$, we define a functor $F_g: \Lambda \text{-mod} \rightarrow \Lambda \text{-mod}$ as follows. For $X \in \Lambda$-mod, $F_g (X) = X^g = X$ as vector spaces. For every $\lambda \in \Lambda$ and $v \in X^g$, the action of $\lambda$ on $v \in X^g$ is defined by $\lambda \ast v = g^{-1} (\lambda) v$. For $f \in \Hom _{\Lambda} (X, Y)$ where $Y \in \Lambda$-mod, $F_g (f) = f^g$, and $f^g(v) = f(v)$ for every $v \in X^g$. To simplify the notation, we identify $f^g$ and $f$ since they are actually the same map.

If $X$ is also a $k G$-module, then $X \cong X^g$ as $\Lambda$-modules. Indeed, we can find a $\Lambda$-module isomorphism $\varphi_g: X \rightarrow X^g$ by letting $\varphi_g (v) = g^{-1} (v)$. This is clearly a vector space isomorphism. For every $\lambda \in \Lambda$,
\begin{equation*}
\varphi_g (\lambda v) = g^{-1} (\lambda v) = g^{-1} (\lambda) g^{-1} (v) = \lambda \ast \varphi_g (v),
\end{equation*}
so it is also a $\Lambda$-module homomorphism.

If both $X$ and $Y$ are $kG$-modules, applying the functor $F_g$ to a $\Lambda$-module homomorphism $f: X \rightarrow Y$, we get the following commutative diagram:
\begin{equation*}
\xymatrix{X \ar[rr] ^{g \ast f} \ar[d] ^{\varphi_g} & & Y \ar[d] ^{\varphi_g} \\
X^g \ar[rr] ^{f^g = f} & & Y^g.}
\end{equation*}
Indeed, for every $v \in X$, we have $f (\varphi_g (v)) = f (g^{-1} (v))$, and
\begin{equation*}
\varphi_g ((g \ast f) (v)) = \varphi_g ( g f (g^{-1} v)) = g^{-1} g f(g^{-1} v) = f (g^{-1} v).
\end{equation*}

\begin{lemma}
Let $M, N$ and $L$ be finitely generated $\Lambda G$-modules. Then there is a natural action of $G$ on $\Ext ^n _{\Lambda} (M, N)$ such that for $x \in \Ext _{\Lambda}^s (M, N)$, $y \in \Ext _{\Lambda}^t (N, L)$, and $g \in G$, $g(y \cdot x) = g(y) \cdot g(x)$, where $\cdot$ is the Yoneda product.
\end{lemma}

\begin{proof}
This the third statement of Lemma 4 in \cite{Martinez}. Note in the original proof the assumptions that $|G|$ is invertible in $k$ and $\Lambda_0$ is semisimple are not used, so the result is true in our framework.
\end{proof}

\begin{lemma}
Let $0 \rightarrow L \rightarrow M \rightarrow N \rightarrow 0$ be an exact sequence of $\Lambda G$-modules and $X$ be a $\Lambda G$-module. Then the long exact sequences
\begin{equation*}
0 \rightarrow \Hom _{\Lambda} (X, L) \rightarrow \Hom _{\Lambda} (X, M) \rightarrow \Hom _{\Lambda} (X, N) \rightarrow \Ext _{\Lambda}^1 (X, L) \rightarrow \ldots
\end{equation*}
and
\begin{equation*}
0 \rightarrow \Hom _{\Lambda} (N, X) \rightarrow \Hom _{\Lambda} (M, X) \rightarrow \Hom _{\Lambda} (L, X) \rightarrow \Ext _{\Lambda}^1 (N, X) \rightarrow \ldots
\end{equation*}
are exact sequences of $kG$-modules.
\end{lemma}

\begin{proof}
This is Corollary 5 in \cite{Martinez}. We have shown that every term is a $kG$-module in the previous lemma, so it is enough to show that every map in the above sequences is $G$-equivariant. The proof of Martinez still works as it does not rely on the assumptions that $|G|$ is invertible in $k$ and $\Lambda_0$ is semisimple.
\end{proof}

The following lemma is Lemma 8 in \cite{Martinez}. As before, The conditions that $|G|$ is invertible in $k$ and $\Lambda_0$ is semisimple are not used in his proof.

\begin{lemma}
Let $P$ be a graded finitely generated projective $\Lambda G$-module, $N$ be a graded finitely generated $\Lambda G$-module, and $W$ be a $kG$-module. Then the linear map
\begin{equation*}
\theta: \Hom _{\Lambda} (P, N) \otimes_k W \rightarrow \Hom _{\Lambda G} (P \otimes_k kG, N \otimes_k W)
\end{equation*}
defined by $\theta (f \otimes w) (v \otimes g) = (g \ast f) (v) \otimes gw$ is a natural isomorphism of $\Lambda G$-modules.
\end{lemma}

The following lemma is similar to Proposition 9 in \cite{Martinez}, with a small modification.

\begin{lemma}
Let $M, N$ be graded $\Lambda G$-modules such that $M$ viewed as a $\Lambda$-module is a generalized Koszul. Then for every $s \geqslant 0$ we have a natural $\Lambda G$-module isomorphism
\begin{equation*}
\tilde{\theta}: \Ext _{\Lambda}^s (M, N) \otimes_k kG \rightarrow \Ext _{\Lambda G} ^s (M \otimes_k kG, N \otimes_k kG).
\end{equation*}
\end{lemma}

\begin{proof}
Take a minimal graded projective resolution for  $_{\Lambda} M$:
\begin{equation*}
\xymatrix{ \ldots \ar[r] & P^n \ar[r] & \ldots \ar[r] & P^0 \ar[r] & _{\Lambda} M \ar[r] & 0}.
\end{equation*}
By Lemma 5.5, every $P^i$ is actually a finitely generated $\Lambda G$-module. Applying the functor $\Hom _{\Lambda} (-, N)$ and $- \otimes_k kG$, we get
\begin{equation*}
0 \rightarrow \Hom _{\Lambda} (P^0, N) \otimes_k kG \rightarrow \ldots \rightarrow \Hom _{\Lambda} (P^s, N) \otimes_k kG \rightarrow \ldots.
\end{equation*}
and deduce that the $s$-th homology is $\Ext _{\Lambda} ^s (M, N) \otimes_k kG$. By the previous lemma, this complex is isomorphic to
\begin{equation*}
0 \rightarrow \Hom _{\Lambda G} (P^0 \otimes_k kG, N \otimes_k kG) \rightarrow \ldots \rightarrow \Hom _{\Lambda G} (P^s \otimes_k kG, N \otimes_k kG) \rightarrow \ldots,
\end{equation*}
by a natural isomorphism $\theta$. Observe that $s$-th homology of the last sequence is $\Ext _{\Lambda G} ^s (M \otimes_k kG, N \otimes_k kG)$. Therefore, $\theta$ determines a $\Lambda G$-module isomorphism $\tilde{\theta}$ on homologies.
\end{proof}

Now we can obtain:

\begin{theorem}
Let $M$ be a graded $\Lambda G$-module and $\Gamma = \Ext _{\Lambda} ^{\ast} (M, M)$.  If as a $\Lambda$-module $M$ it is generalized Koszul, then $\Ext _{\Lambda \ast G} ^{\ast} (M \otimes_k kG, M \otimes_k kG)$ is isomorphic to the skew group algebra $\Gamma G$.
\end{theorem}

\begin{proof}
This is Theorem 10 in \cite{Martinez}.
\end{proof}

\end{document}